\documentclass[11pt,oneside,american,english]{amsart}
\usepackage[latin9]{inputenc}
\usepackage{geometry}
\usepackage{color}
\usepackage{amstext}
\usepackage{amsthm}
\usepackage{amssymb}
\usepackage{graphicx}
\usepackage{esint}

\makeatletter

\DeclareFontEncoding{LGR}{}{}
\DeclareTextSymbol{\~}{LGR}{126}

\numberwithin{equation}{section}
\numberwithin{figure}{section}
\theoremstyle{plain}
\newtheorem{thm}{\protect\theoremname}
  \theoremstyle{definition}
  \newtheorem{example}[thm]{\protect\examplename}
  \theoremstyle{definition}
  \newtheorem{defn}[thm]{\protect\definitionname}
  \theoremstyle{plain}
  \newtheorem{cor}[thm]{\protect\corollaryname}
  \theoremstyle{plain}
  \newtheorem{prop}[thm]{\protect\propositionname}
  \theoremstyle{remark}
  \newtheorem*{rem*}{\protect\remarkname}
  \theoremstyle{plain}
  \newtheorem{lem}[thm]{\protect\lemmaname}
  \theoremstyle{remark}
  \newtheorem{rem}[thm]{\protect\remarkname}

\makeatother

\usepackage{babel}
  \addto\captionsamerican{\renewcommand{\corollaryname}{Corollary}}
  \addto\captionsamerican{\renewcommand{\definitionname}{Definition}}
  \addto\captionsamerican{\renewcommand{\examplename}{Example}}
  \addto\captionsamerican{\renewcommand{\lemmaname}{Lemma}}
  \addto\captionsamerican{\renewcommand{\propositionname}{Proposition}}
  \addto\captionsamerican{\renewcommand{\remarkname}{Remark}}
  \addto\captionsamerican{\renewcommand{\theoremname}{Theorem}}
  \addto\captionsenglish{\renewcommand{\corollaryname}{Corollary}}
  \addto\captionsenglish{\renewcommand{\definitionname}{Definition}}
  \addto\captionsenglish{\renewcommand{\examplename}{Example}}
  \addto\captionsenglish{\renewcommand{\lemmaname}{Lemma}}
  \addto\captionsenglish{\renewcommand{\propositionname}{Proposition}}
  \addto\captionsenglish{\renewcommand{\remarkname}{Remark}}
  \addto\captionsenglish{\renewcommand{\theoremname}{Theorem}}
  \providecommand{\corollaryname}{Corollary}
  \providecommand{\definitionname}{Definition}
  \providecommand{\examplename}{Example}
  \providecommand{\lemmaname}{Lemma}
  \providecommand{\propositionname}{Proposition}
  \providecommand{\remarkname}{Remark}
\providecommand{\theoremname}{Theorem}

\begin{document}
\global\long\def\bbZ{\mathbb{Z}}

\global\long\def\RR{\mathbb{R}}

\global\long\def\bbN{\mathbb{N}}

\global\long\def\bbX{\mathbb{X}}

\global\long\def\MM{\mathbb{M}}

\global\long\def\cF{\mathcal{F}}

\global\long\def\cB{\mathcal{B}}

\global\long\def\fT{\mathfrak{\tau}}

\global\long\def\cCT{\mathcal{C}_{\fT-1}}

\global\long\def\cC{\mathcal{C}}

\global\long\def\RR{\mathbb{R}}

\global\long\def\BB{\mathcal{B}}

\global\long\def\P{\mathrm{P}}

\global\long\def\C#1{\mathcal{#1}}

\global\long\def\ZZ{\mathbb{Z}}

\global\long\def\RN#1{\left(T^{#1}\right)'}

\global\long\def\F#1{\mathcal{F}_{#1}}

\global\long\def\lloc{\ll^{loc}}

\global\long\def\PP{\mathbb{P}}

\global\long\def\ep{\underline{\epsilon}}

\global\long\def\p{\varphi}

\global\long\def\l{\lambda}

\global\long\def\NN{\mathbb{N}}

\global\long\def\M#1{\text{M}\left\{  #1\right\}  }

\global\long\def\P{{\rm P}}

\global\long\def\TT{\mathbb{T}}

\global\long\def\g{\mathfrak{g}}

\global\long\def\h{\mathfrak{h}}

\global\long\def\r{\varrho}

\global\long\def\qq{\mathtt{q}}

\global\long\def\rr{\mathtt{r}}

\title{Examples of type ${\rm III}_{1}$ inhomogenous Markov shifts supported
on topological Markov shifts}

\author{Zemer Kosloff}\thanks{The research of Z.K. was supported in part by the European Advanced
Grant StochExtHomog (ERC AdG 320977).}
\address{Mathematics Institute, University of Warwick, Coventry CV4 7AL, UK}
\email{z.kosloff@warwick.ac.uk}

\address{Email: z.kosloff@warwick.ac.uk}
\begin{abstract}
We construct inhomogenous Markov measures for which the shift is of
Kreiger type ${\rm III}_{1}$. These measures are fully supported
on a toplogical Markov shift space of the hyperbolic toral automorphism
$f(x,y)=\left(\{x+y\},x\right):\TT^{2}\to\TT^{2}$ and are used in
a subsequent paper \cite{Kos1} in the construction of $C^{1}$ Anosov
diffeomorphisms of $\mathbb{T}^{2}$ without an absolutely continuous
invariant measure. 
\end{abstract}

\maketitle

\section{Introduction}

Topological markov shifts (TMS) are an important model in ergodic
theory which plays a central role in smooth dynamics of $C^{1+\alpha}$
hyperbolic diffeomorphisms. Indeed, the construction of Markov partitions
of Adler and Weiss \cite{Adler Weiss} (toral automorphisms), Sinai
\cite{Si} (Anosov diffeomorphisms) and Bowen \cite{Bowen Rufus}
(Axion A diffeomorphisms) shows that given an hyperbolic diffeomorphism
there exists a TMS and a finite to one factor map from the TMS $\Sigma$
to the manifold $M$ which intertwines the actions of the shift $T$
and the diffeomorphism $f$. This factor map and the symbolic model
were used extensively in smooth dynamics, examples include showing
existence and uniqueness of the measure of maximal entropy and equilibrium
states, proving central limit theorems, Livsic' type theorems and
more. 

Another feature of the factor map $\pi:\Sigma\to M$ is that it is
finite to one except on the (volume) measure zero set which consists
of the union of iterations of $f$ on the (topological) boundary of
the atoms of the Markov partition, see example \ref{Example of a TMS}
below. This shows that a large class of measures on $M$ for which
$f$ is non singular (also known as quasi invariant), which includes
all the measures for which $f$ satisfies Poincare recurrence, can
be lifted back (since $\pi^{-1}$ is a function on the support of
the measure) to the TMS. This paper deals with constructions of inhomogenous
markov measures on TMS for which the shift is non singular, conservative
(recurrent) and ergodic yet there exists no absolutely continuous
shift invariant measure. In fact the system $\left(\Sigma,\BB_{\Sigma},\mu,T\right)$
is of Krieger type ${\rm III}_{1}$. These examples are Markovian
analogues of the Bernoulli shifts with non identical factor measures
as in \cite{Ham,Krengel,Kos}. 

The main motivation for us in this work is the paper \cite{Kos1}
where these measures are used in the contruction of a conservative,
ergodic $C^{1}$ Anosov diffeomorphism of $\mathbb{T}^{2}$ without
a Lebesgue absolutely continuous invariant measure. These diffeomorphisms
are obtained by smooth realization of $\left(\Sigma,\BB_{\Sigma},\mu,T\right)$. 

This paper is organised as follows. Section \ref{sec:Prelminary-definition-and}
contains the basic definitions from ergodic theory and the probability
theory of Markov chains which will be used in this paper. Section
\ref{sec:Type--Markov} starts with a discussion on the condition
of non singularity for the shift map with respect to a half stationary
inhomogenous Markov measure and a suffiecient condition for exactness
of the one sided shift. After this we present the construction of
the measures and prove that the shift is of type ${\rm III}_{1}$
with respect to these measures. We end the paper with an explanation
on how to construct such measures on a general TMS, some remarks about
these measures and an open question on the possible Krieger types
of Markov shifts.

\section{\label{sec:Prelminary-definition-and}Preliminary definitions }

\subsection{Short introduction to non singular ergodic theory}

For more details and explanations we refer the reader to \cite{Aar}. 

Let $\left(X,\BB,\mu\right)$ be a standard probability space. In
what follows equalities (and inclusions) of sets are modulo the measure
$\mu$ on the space. A measurable map $T:X\to X$ is \textit{non singular}
if $T_{*}\mu:=\mu\circ T^{-1}$ is equivalent to $\mu$ meaning that
they have the same collection of negligible sets. If $T$ is invertible
one has the Radon Nykodym derivatives 
\[
\left(T^{n}\right)'(x):=\frac{d\mu\circ T^{n}}{d\mu}(x):\ X\to\RR_{+},
\]
 A set $W\subset X$ is \textit{wandering} if $\left\{ T^{n}W\right\} _{n\in\ZZ}$
are pairwise disjoint and $T$ is \textit{conservative} if there exists
no wandering set of positive measure. By Hopf's criteria, an invertible
$T$ is conservative (w.r.t $\mu$) if and only if 
\[
\sum_{n=1}^{\infty}\left(T^{n}\right)'(x)=\infty\ \mu-a.e.
\]

A transformation $T$ is \textit{ergodic} if there are no non trivial
$T$ invariant sets. That is $T^{-1}A=A$ implies $A\in\{\emptyset,X\}$.
If $\left(X,\BB,\mu\right)$ is a non atomic measure space and $T$
is invertible and ergodic then $T$ is also conservative. The converse
implication is not true in general as there are conservative non-ergodic
transformations. Since proving ergodicity is usually a harder task
then proving conservativity we would like to concentrate on a class
of transformations, called $K$-automorphisms, for which conservativity
implies ergodicity.

A transformation is a $K$-\textit{\textcolor{black}{automorphism
}}if there exists a $\sigma$-algebra $\mathcal{F\subset B}$ such
that:
\begin{itemize}
\item $T^{-1}\cF\subset\BB$ meaning that $\cF$ is a factor of $\BB$. 
\item $\cap_{n=1}^{\infty}T^{-k}\cF=\left\{ \emptyset,X\right\} $ (exactness)
and $\vee_{n\in\ZZ}T^{n}\cF=\BB$ (exhaustiveness). 
\item $T'(x)$ is $\cF$ measurable. 
\end{itemize}
The first two properties are the standard definition of a $K$-automorphism
in the case of measure preserving automorphisms. The condition on
the measurability of the Radon Nykodym derivative comes to ensure
that $\left(X,\BB,\mu,T\right)$ is the \textbf{unique }natural extension
of the non invertible exact transformation $\left(X/\cF,\cF,\mu|_{\cF},T\right)$.
It was shown in \cite{Krengel,S-T} that a conservative and $K$ transformation
is necessarily ergodic.

The \textit{Krieger ratio set} $R(T)$: We say that $r\geq0$ is in
$R(T)$ if for every $A\in\BB$ of positive $\mu$ measure and for
every $\epsilon>0$ there exists an $n\in\ZZ$ such that 
\[
\mu\left(A\cap T^{-n}A\cap\left\{ x\in X:\left|\left(T^{n}\right)'(x)-r\right|<\epsilon\right\} \right)>0.
\]
The ratio set of an ergodic measure preserving transformation is a
closed multiplicative subgroup of $[0,\infty)$ and hence it is of
the form $\{0\},\{1\},\{0,1\},\{0\}\cap\left\{ \lambda^{n}:n\in\ZZ\right\} $
for $0<\lambda<1$ or $[0,\infty)$. Several ergodic theoretic properties
can be seen from the ratio set. One of them is that $0\in R(T)$ if
and only if there exists no $\sigma$-finite $T$-invariant $\mu$-
a.c.i.m. Another interesting relation is that $1\in R(T)$ if and
only if $T$ is conservative (Maharam's Theorem). If $R(T)=[0,\infty)$
we say that $T$ is of type ${\rm III}_{1}$.

\subsubsection{Topological Markov shifts}

A \textit{topological Markov shift} (TMS) on $S$ is the shift on
a shift invariant subset $\Sigma\subset S^{\mathbb{Z}}$ of the form 

\[
\Sigma_{A}:=\left\{ x\in S^{\ZZ}:A_{x_{i},x_{i+1}}=1\right\} ,
\]
where $A=\left\{ A_{s,t}\right\} _{s,t\in S}$ is a $\{0,1\}$ valued
matrix on $S$. A TMS is mixing if there exists $n\in\NN$ such that
$A_{s,t}^{n}>0$ for every $s,t\in S$. 

TMS appear in ergodic theory as a symbolic model for $C^{1+\alpha}$
Anosov and Axiom A diffeomorphisms via the construction Markov partitions
of the manifold $M$ \cite{Adler Weiss,Si,Bowen Rufus,Adler}. We
present here an example which motivates the choice of the TMS we will
be performing the construction on. 
\begin{example}
\label{Example of a TMS}Consider $f:\mathbb{T}^{2}\to\mathbb{T}^{2}$
the Toral automorphism defined by 
\[
f(x,y)=(\{x+y\},x)=\left(\begin{array}{cc}
1 & 1\\
1 & 0
\end{array}\right)\binom{x}{y}\ \mod1,
\]
where $\{t\}$ is the fractional part of $t$. Since $\left|\det\left(\begin{array}{cc}
1 & 1\\
1 & 0
\end{array}\right)\right|=1$, $f$ preserves the Lebesgue measure on $\mathbb{T}^{2}$. In \cite{Adler Weiss},
a method is given for the construction of Markov partitions for hyperbolic
toral automorphisms. For this example one has a a Markov partition
with three elements $\{R_{1},R_{2},R_{3}\}$, see figure \ref{fig:The-construction-of}
and \cite{Adler}.
\end{example}
\begin{figure}
\includegraphics[scale=0.5]{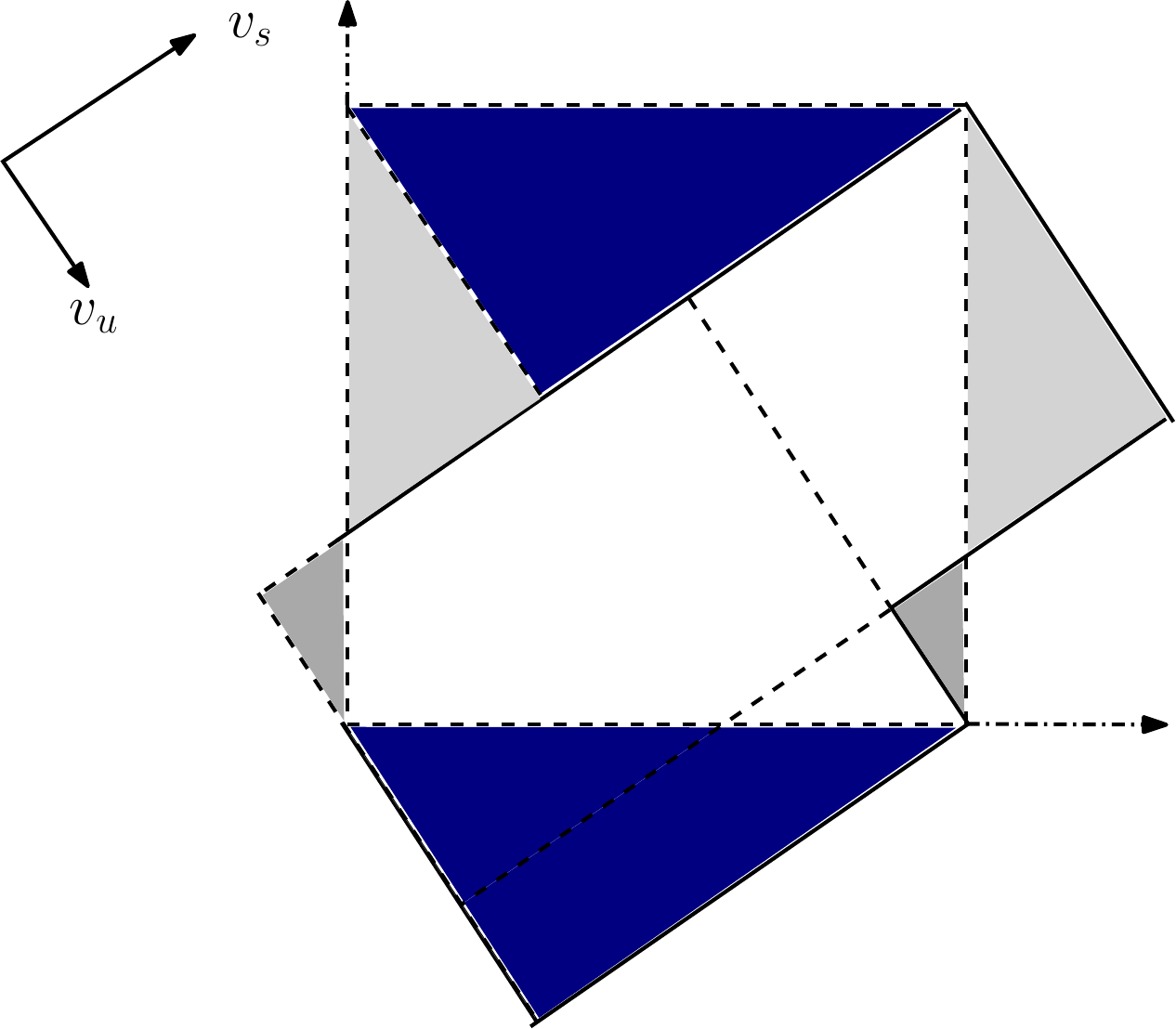}\qquad{}\includegraphics[scale=0.65]{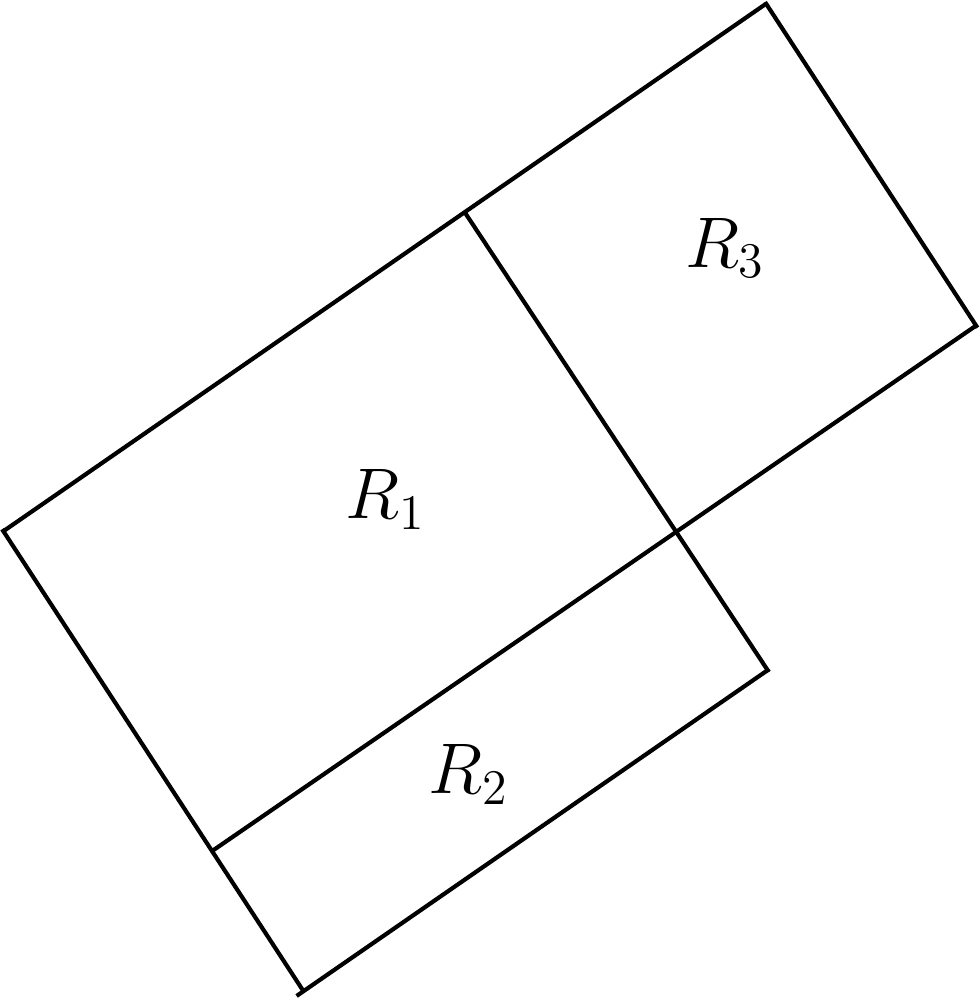}\caption{The construction of t\label{fig:The-construction-of}he Markov partition}
\end{figure}

The adjacency Matrix of the Markov partition is then defined by $A_{i,j}=1$
if and only if $R_{i}\cap f^{-1}\left(R_{j}\right)\neq\emptyset$.
In this example, 
\[
{\bf A}=\left(\begin{array}{ccc}
1 & 0 & 1\\
1 & 0 & 1\\
0 & 1 & 0
\end{array}\right).
\]
Let $\Phi:\Sigma_{A}\to\mathbb{T}^{2}$ be the map defined by 
\[
\Phi(x):=\bigcap_{n=-\infty}^{\infty}\overline{f^{-n}R_{x_{n}}}.
\]
Note that $\left\{ \cap_{n=-N}^{N}\overline{f^{-n}R_{x_{n}}}\right\} _{N=1}^{\infty}$
is a decreasing sequence of compact sets, hence by the Baire Category
Theorem, $\Phi(x)$ is well defined. The map $\Phi:\Sigma_{{\bf A}}\to\mathbb{T}^{2}$
is continuous, finite to one, and for every $x\in\Sigma_{{\bf A}}$,
\[
\Phi\circ T(x)=f\circ\Phi(x).
\]
 Thus $\Phi$ is a semi-conjugacy (topological factor map) between
$\left(\Sigma_{{\bf A}},T\right)$ to $\left(\mathbb{T}^{2},f\right)$.
In addition, for every $x\in\TT^{2}\backslash\cup_{n\in\ZZ}\cup_{i=1}^{3}f^{-n}\left(\partial R_{i}\right)$
there exists a unique $w\in\Sigma_{{\bf A}}$ so that $w=\Phi^{-1}(x).$
The Lebesgue measure $\lambda$ on $\mathbb{T}^{2}$ is $f$ invariant
and $\lambda\left(\cup_{n\in\ZZ}\cup_{i=1}^{3}f^{-n}\left(\partial R_{i}\right)\right)=0$.
Thus $\Phi$ defines an isomorphism between $\left(\mathbb{T}^{2},\lambda,f\right)$
and $\left(\Sigma_{{\bf A}},\mu_{\pi_{{\bf Q}},{\bf Q}},T\right)$
where $\mu_{\pi_{{\bf Q}},{\bf Q}}=\Phi_{*}\lambda$ is the stationary
Markov measure with 
\begin{equation}
P_{j}\equiv\mathbf{Q}:=\left(\begin{array}{ccc}
\frac{\varphi}{1+\varphi} & 0 & \frac{1}{1+\varphi}\\
\frac{\varphi}{1+\varphi} & 0 & \frac{1}{1+\varphi}\\
0 & 1 & 0
\end{array}\right)\label{eq: Lebesgue Transition function}
\end{equation}
and 
\begin{equation}
\pi_{j}=\mathbf{{\bf \pi_{Q}}}:=\left(\begin{array}{c}
1/\sqrt{5}\\
1/\varphi\sqrt{5}\\
1/\varphi\sqrt{5}
\end{array}\right)=\left(\begin{array}{c}
\lambda\left(R_{1}\right)\\
\lambda\left(R_{2}\right)\\
\lambda\left(R_{3}\right)
\end{array}\right).\label{eq:Lebesgue stationary dist.}
\end{equation}

\vspace{-3.5mm}

\subsection{\label{sub:Markov-Chains}Markov Chains}

\subsubsection{Basics of Stationary (homogenous) Chains}

Let $S$ be a finite set which we regard as the state space of the
chain, $\pi=\left\{ \pi(s)\right\} _{s\in S}$ a probability vector
on $S$ and ${\rm P}=\left(P_{s,t}\right)_{s,t\in S}$ a stochastic
matrix. The vector $\pi$ and $\P$ define a Markov chain $\left\{ X_{n}\right\} $
on $S$ by 
\begin{eqnarray*}
\PP_{\pi}\left(X_{0}=t\right) & := & \pi(t)\\
\PP\left(X_{n}=s\left|X_{1},..,X_{n-1}\right.\right) & := & P_{X_{n-1},s}.
\end{eqnarray*}
${\rm P}$ is \textit{irreducible} if for every $s,t\in S$, there
exists $n\in\NN$ such that 
\[
P_{s,t}^{n}=\PP\left(X_{n}=t\left|X_{0}=s\right.\right)>0,
\]
and ${\rm P}$ is \textit{aperiodic} if for every $s\in S$ 
\[
{\rm gcd}\left\{ n:\ P_{s,s}^{n}>0\right\} =1.
\]
Given an irreducible and aperiodic ${\rm P}$, there exists a unique
stationary ($\pi_{\mathrm{P}}\mathrm{P}=\pi_{\mathrm{P}}$) probability
vector $\pi_{{\rm P}}$. In addition for every $s,t\in S$, 
\[
P_{s,t}^{n}\xrightarrow[n\to\infty]{}\pi_{P}(t).
\]
Since $S$ is a finite state space, it follows that for any initial
distribution $\pi$ on $S$, 
\[
\PP_{\pi}\left(X_{n}=t\right)=\sum_{s\in S}\pi(s)P_{s,t}^{n}\xrightarrow[n\to\infty]{}\pi_{_{P}}(t).
\]
An important fact which will be used in the sequel is that the stationary
distribution is continuous with respect to the stochastic matrix.
That is if $\left\{ {\rm P_{n}}\right\} _{n=1}^{\infty}$ is a sequence
of irreducible and aperiodic stochastic matrices such that 
\[
\|{\rm P}_{n}-{\rm P\|_{\infty}:=}\max_{s,t\in S}\left|\left(\P_{n}\right)_{s,t}-\P_{s,t}\right|\xrightarrow[n\to\infty]{}0
\]
and ${\rm P}$ is irreducible and aperiodic then 
\[
\left\Vert \pi_{_{{\rm P}_{n}}}-\pi_{_{{\rm P}}}\right\Vert _{\infty}\to0.
\]

\subsubsection{Non singular Markov shifts: }

Let $S$ be a finite set. An inhomogeneous Markov Chain is a stochastic
process $\left\{ X_{n}\right\} _{n\in\ZZ}$ such that for each times
$t_{1},..,t_{l}\in\ZZ$, and $s_{1},\ldots,s_{l}\in S$,
\[
\mathbb{P}\left(X_{t_{1}}=s_{1},X_{t_{2}}=s_{2},\dots,X_{t_{l}}=s_{l}\right)=\mathbb{P}\left(X_{t_{1}}=s_{1}\right)\prod_{k=1}^{l-1}\mathbb{P}\left(\left.X_{t_{k+1}}=s_{k+1}\right|X_{t_{k}}=s_{k}\right).
\]
Note that unlike in the classical setting of Markov Chains $\mathbb{P}\left(\left.X_{t_{k+1}}=s_{k+1}\right|X_{t_{k}}=s_{k}\right)$
can depend on $t_{k}$. 

The ergodic theoretical formulation is as follows. Let $\left\{ P_{n}\right\} _{n=-\infty}^{\infty}\subset M_{S\times S}$
be a sequence of stochastic matrices on $S.$ In addition let $\left\{ \pi_{n}\right\} _{n=-\infty}^{\infty}$
be a sequence of probability distributions on $S$ so that for every
$s\in S$ and $n\in\ZZ$,
\begin{equation}
\sum_{t\in S}\pi_{n-1}(t)\cdot P_{n}\left(t,s\right)=\pi_{n}(s).\label{eq: Kolmogorov's consistency criteria}
\end{equation}
Then one can define a measure on the collection of cylinder sets,
\[
[b]_{k}^{l}:=\left\{ x\in S^{\ZZ}:\ x_{j}=b_{j}\ \forall j\in[k,l]\cap\ZZ\right\} 
\]
by 
\[
\mu\left(\left[b\right]_{k}^{l}\right):=\pi_{k}\left(b_{k}\right)\prod_{j=k}^{l-1}P_{j}\left(b_{j},b_{j+1}\right).
\]
Since the equation (\ref{eq: Kolmogorov's consistency criteria})
is satisfied, $\mu$ satisfies the consistency condition. Therefore
by Kolmogorov's extension theorem $\mu$ defines a measure on $S^{\ZZ}$.
In this case we say that $\mu$ is the Markov measure generated by
$\left\{ \pi_{n},P_{n}\right\} _{n\in\ZZ}$ and denote $\mu=\text{M}\left\{ \pi_{n},P_{n}:\ n\in\ZZ\right\} $.
By $\text{M}\left\{ \pi,P\right\} $ we mean the measure generated
by $P_{n}\equiv P$ and $\pi_{n}\equiv\pi$. We say that $\mu$ is
non singular for the shift $T$ on $S^{\ZZ}$ if $T_{*}\mu\sim\mu$.
See subsection \ref{sub:Non-singular-Markov shifts} for an extension
of Kakutani's Theorem for product measures to a class of inhomogenous
Markov measures which is used in giving a crtieria for non singularity
of the shift in our examples.

\section{\label{sec:Type--Markov}Type ${\rm {\rm III}_{1}}$ Markov shifts
supported on topological Markov shifts}

This section is organized as follows. First we give a condition for
non singularity of the Markov measure which is an application of ideas
of Cabanos, Liptzer and Shiryaev \cite{Shi,Le Page Mandrelkar } and
prove a necessary condition for exactness of a one sided shift. Then
we construct the aforementioned examples and prove that they are type
${\rm III}_{1}$ measures for the shift.

\subsubsection{\label{sub:Non-singular-Markov shifts} Non Singularity criteria
for Markov shifts}

In order to check if a measure is shift non singular we apply the
following reasoning of \cite{Shi}, see also \cite{Le Page Mandrelkar }. 
\begin{defn}
Given a filtration $\left\{ \F n\right\} $, we say that $\nu\ll^{loc}\mu$
($\nu$ is locally absolutely continuous with respect to $\mu$) if
for every $n\in\mathbb{N}$ 
\[
\nu_{n}\ll\mu_{n}
\]
 where 
\[
\nu_{n}=\nu|_{\F n}.
\]
 
\end{defn}
Suppose that $\nu\lloc\mu$ w.r.t $\left\{ \F n\right\} $, set 
\[
z_{n}:=\frac{d\nu_{n}}{d\mu_{n}},
\]
and 
\[
\alpha_{n}(x):=z_{n}(x)\cdot z_{n-1}^{\oplus}(x),
\]
where $z_{n-1}^{\oplus}=\frac{1}{z_{n-1}}\cdot{\bf 1}_{\left[z_{n-1}\neq0\right]}$.
The question is when $\nu\lloc\mu$ implies $\nu\ll\mu$. 
\begin{thm}
\label{thm: Shiryaev}\cite[Thm. 4, p. 528]{Shi}. If $\nu\lloc\mu$
then $\nu\ll\mu$ if and only if 
\[
\sum_{k=1}^{\infty}\left[1-E_{\mu}\left(\left.\sqrt{\alpha_{n}}\right|\F{n-1}\right)\right]<\infty\quad\nu\ a.s.
\]
If $\nu\ll\mu$ then 
\[
\frac{d\nu}{d\mu}=\lim_{n\to\infty}z_{n}.
\]

\end{thm}
Given a a markovian measure $\mu=\M{\pi_{n},P_{n}:\ n\in\ZZ}$ on
$S^{\ZZ}$, we want to know when $\mu\sim\mu\circ T$. The natural
filtration on the product space is the sequence of algebras $\C F_{n}:=\sigma\left\{ [b]_{-n}^{n};\ b\in S^{\ZZ\cap[-n,n]}\right\} $. 

The measures which we will construct are fully supported on a TMS
$\Sigma_{{\bf A}}$ meanning that for every $n\in\NN$, 
\[
\text{supp}P_{n}:=\left\{ (s,t)\in S\times S:\ P_{n}(s,t)>0\right\} =\text{supp}{\bf A}.
\]
This implies that $\mu\circ T\lloc\mu$. In addition the measure $\M{P_{n},\pi_{n}}$
will be half stationary in the sense that for every $j\leq0$, $P_{j}:=\mathbf{Q}$
where ${\bf Q}$ is the matrix from Example \ref{Example of a TMS}.
This condition implies that for every $j\leq0$, $\pi_{j}=\pi_{{\bf _{Q}}}$
and 
\[
\alpha_{n}(x)=\frac{P_{n-1}\left(x_{n-1},x_{n}\right)}{P_{n}\left(x_{n-1},x_{n}\right)},\ \forall n>0.
\]
By Theorem \ref{thm: Shiryaev}, in this setting, $\mu-$non singular
if and only if 
\[
\sum_{n=-\infty}^{\infty}\left[1-E_{\nu}\left(\left.\sqrt{\alpha_{n}}\right|\F{n-1}\right)(x)\right]=\sum_{n=0}^{\infty}\left[1-\sum_{s\in S}\sqrt{P_{n-1}\left(x_{n-1},s\right)P_{n}\left(x_{n-1},s\right)}\right]<\infty
\]
for $\mu\circ T\ a.e.\ x.$ The following corollary concludes our
discussion. 
\begin{cor}
\label{Non singularity for Markov shifts}Let $\nu=\M{\pi_{n},P_{n}:\ n\in\ZZ}$,
where $\left\{ P_{n}\right\} $ are fully supported on a TMS $\Sigma_{A}$
and there exists an aperiodic and irreducible $P\in M_{\mathcal{S\times S}}$
such that for all $n\leq0$, $P_{n}\equiv P$

\begin{itemize} 

\item   $\nu\circ T\sim\nu$ if and only if 

\begin{equation}
\sum_{n=0}^{\infty}\sum_{s\in S}\left(\sqrt{P_{n}\left(x_{n},s\right)}-\sqrt{P_{n-1}\left(x_{n},s\right)}\right)^{2}<\infty,\ \nu\circ T\ a.s.\ x.\label{eq: Non singularity condition.}
\end{equation}
 \item   If $\nu\circ T\sim\nu$ then for all $n\in\NN$, 
\[
\RN{{\bf n}}(x)=\prod_{k=0}^{\infty}\frac{P_{k-{\bf n}}\left(x_{k},x_{k+1}\right)}{P_{k}\left(x_{k},x_{k+1}\right)}.
\]
\end{itemize}  
\end{cor}

\subsubsection*{A condition for exactness of the one sided shift}

Let $S$ be a countable set and $\left\{ \left(\pi_{n},P_{n}\right)\right\} _{n=1}^{\infty}\subset\mathcal{P}(S)\times\mathcal{M}_{S\times S}$.
Denote the one sided shift on $S^{\NN}$ by $\sigma$and by $\cF$
the Borel $\sigma-$algebra of $S^{\NN}$. The following is a sufficient
condition for exactness (trivial tail $\sigma$-field) of the one
sided shift which is well known in the theory of non homogenous Markov
chains.\textcolor{red}{{} }\textcolor{black}{We include a simple ergodic
theoretic proof for the sake of completeness. }
\selectlanguage{american}%
\begin{prop}
\label{prop:Exactness of Markov shifts}Let $S$ be a countable set
and $\mu$ be a Markovian measure on $S^{\mathbb{N}}$ which is defined
by $\left\{ \pi^{(k)},P^{(k)}:\ k\in\NN\cup\{0\}\right\} $. If there
exists $C>0$ and $N_{0}\in\mathbb{N}$ so that for every $s,t\in S$,
and $k\in\mathbb{N}$,
\begin{equation}
\left(P_{k}P_{k+1}\cdots P_{k+N_{0}-1}\right)_{s,t}\geq C\label{eq:NS Aperidoicity}
\end{equation}
 then the one sided shift $\left(S^{\mathbb{N}\cup\{0\}},\mathcal{F},\mu,\sigma\right)$
is exact. \end{prop}
\selectlanguage{english}%
\begin{rem*}
In the setting of Markov maps, exactness was proved under various
distortion properties {[}see \cite{Aar,Thaler}{]}. Their conditions
guarantees the existence of an absolutely continuous $\sigma$-finite
invariant measure. \end{rem*}
\selectlanguage{american}%
\begin{proof}
The measure $\mu\circ T^{-n}$ is the Markov measure generated by
$Q_{k}:=P_{k+n}$ and $\tilde{\pi}_{k}:=\pi_{k+n}$. Let $\alpha_{n}$
be the collection of $n$ cylinders of the form $[d]_{1}^{n}$ and
$\alpha^{*}=\cup_{n}\alpha_{n}$.

For every $D=\left[a\right]_{0}^{n}\in\alpha_{n}$ and $B=[b]_{0}^{n(B)}\in\alpha^{*}$,
\begin{eqnarray*}
\mu\left(D\cap T^{-\left(n+N_{0}\right)}B\right) & = & \mu(D)\left(P_{n}P_{n+1}\cdots P_{n+N_{0}-1}\right)_{a_{n},b_{0}}\prod_{j=0}^{n(B)-1}P_{N_{0}+n+j}\left(b_{j},b_{j+1}\right)\\
 & \geq & C\mu(D)\pi_{n+N_{0}}\left(b_{0}\right)\prod_{j=0}^{n(B)-1}P_{N_{0}+n+j}\left(b_{j},b_{j+1}\right)=C\cdot\mu(D)\mu\circ T^{-\left(n+N_{0}\right)}(B).
\end{eqnarray*}
Consequently for all $B\in\cF$ and $D\in\alpha_{n}$, 
\[
\mu\left(D\cap T^{-\left(n+N_{0}\right)}B\right)\geq C\cdot\mu(D)\mu\circ T^{-\left(n+N_{0}\right)}(B).
\]
Let $B\in\cap_{n=1}^{\infty}\sigma^{-n}\cF$ and $D\in\alpha_{n}$.
Writing $B_{n+N_{0}}\in\cF$ for a set such that $B=T^{-n-N_{0}}B_{n}$,
\begin{eqnarray*}
\mu\left(D\cap B\right) & = & \mu\left(D\cap T^{-\left(n+N_{0}\right)}B_{n+N_{0}}\right)\\
 & \geq & C\cdot\mu(D)\mu\circ T^{-\left(n+N_{0}\right)}\left(B_{n+N_{0}}\right)\\
 & = & C\mu(D)\mu(B)
\end{eqnarray*}
and thus for every $n\in\mathbb{N}$, $\mu\left(B\left|\alpha_{n}\right.\right)\geq C\mu(B).$
Since $\alpha_{n}\uparrow\alpha^{*}$ and $\alpha^{*}$ generates
$\cF$, 
\[
\mu\left(B\left|\alpha_{n}\right.\right)(x)\xrightarrow[n\to\infty]{}1_{B}(x)\ \mu-a.s.
\]
by the Martingale convergence theorem. It follows that if $\mu(B)>0$
then 
\[
1_{B}(x)\geq C\mu(B)>0\ \mu-a.s.
\]
This shows that for every $B\in\cap_{n=1}^{\infty}\sigma^{-n}\cF$,
$\mu(B)\in\{0,1\}$ (the shift is exact).
\end{proof}
\selectlanguage{english}%

\subsection{\label{sub:Type III_1 Markov shifts}Type ${\rm III_{1}}$ Markov
Shifts}

We will construct a Markov measure supported on $\Sigma_{{\bf A}}$
from Example \ref{Example of a TMS}. In the end of this subsection
we will explain what needs to be altered in the case of a general
mixing TMS and conclude with some open questions. 

In this subsection, let $\Omega:=\Sigma_{{\bf A}}$, $\mathcal{B}:=\mathcal{B}_{\Sigma_{{\bf A}}}$
and $T$ is the two sided shift on $\Omega$. For two integers $k<l$,
write $\mathcal{F}\left(k,l\right)$ for the algebra of sets generated
by cylinders of the form $[b]_{k}^{l},\ b\in\{1,2,3\}^{l-k}$.

\subsubsection{Idea of the construction of the type ${\rm III}$ Markov measure. }

The construction uses the ideas in \cite{Kos1}. For every $j\leq0$
\[
P_{j}\equiv\mathbf{Q}\ {\rm and}\ \pi_{j}\equiv\pi_{\mathbf{_{Q}}},
\]
where $\mathbf{Q}$ and $\pi_{\mathbf{Q}}$ are as in (\ref{eq: Lebesgue Transition function})
and (\ref{eq:Lebesgue stationary dist.}) respectively. On the positive
axis one defines on larger and larger chunks the stochastic matrices
which depend on a distortion parameter $\lambda_{k}\geq1$ where $1$
means no distortion. Now a cylinder set $[b]_{-n}^{n}$ fixes the
values of the first $n$ terms in the product form of the Radon Nykodym
derivatives. As a natural first step in proving that a given number
is in the ratio set is to check the condition for cylinder sets, we
would like to be able to correct the values of the Radon Nykodym derivatives
which were fixed by the cylinder set. This corresponds to a lattice
condition on $\lambda_{k}$ which is less straightforward then the
one in \cite{Kos}. The measure of the set $[b]_{-n}^{n}\cap T^{-N}[b]_{-n}^{n}\cap\left\{ \left(T^{N}\right)'\approx a\right\} $
could be very small which forces us to look for many approximately
independent such events so that their union covers at least a fixed
proportion of $[b]_{-n}^{n}$. In the product measure case we used
independence of the coordinates. For a Markov measure the coordinates
are not independent and this leads us to a condition which comes from
the convergence to the stationary distribution and the mixing property
for stationary chains. 

More specifically the construction goes as follows. We define inductively
$5$ sequences $\left\{ \lambda_{j}\right\} $, $\left\{ m_{j}\right\} $,
$\left\{ n_{j}\right\} $, $\left\{ N_{j}\right\} $ and $\left\{ M_{j}\right\} $
where 
\begin{eqnarray*}
M_{0} & = & 1\\
N_{j} & := & N_{j-1}+n_{j}\\
M_{j} & := & N_{j}+m_{j}.
\end{eqnarray*}
 This defines a partition of $\mathbb{N}$ into segments $\left\{ \left[M_{j-1},N_{j}\right),\ \left[N_{j},M_{j}\right)\right\} _{j=1}^{\infty}$.
The sequence $\left\{ P_{n}\right\} $ equals $\mathbf{Q}$ on the
$\left[N_{j},M_{j}\right)$ segments while on the $\left[M_{j-1},N_{j}\right)$
segments we have $P_{n}\equiv\mathbf{Q}_{\lambda_{j}}$, the $\lambda_{j}$
perturbed stochastic matrix. The $\mathbf{Q}$ segments facilitate
the form of some of the Radon Nykodym derivatives while the perturbed
segments come to ensure that $\mu\perp\M{\pi_{{\bf Q}},{\bf Q}}$
and that the ratio set condition is satisfied for cylinder sets.

Notation: By $x=a\pm b$ we mean $a-b\leq x\leq a+b$.

\subsubsection{The construction.\label{sub:The-construction.}}

Choice of the base of induction: Let $M_{0}=1,\ \lambda_{1}>1$, $n_{1}=2$,
$N_{1}=3$ and 
\[
\mathbf{Q}_{\mathbf{1}}:=\left(\begin{array}{ccc}
\frac{\lambda_{1}\varphi}{1+\lambda_{1}\varphi} & 0 & \frac{1}{1+\lambda_{1}\varphi}\\
\frac{\varphi}{1+\varphi} & 0 & \frac{1}{1+\varphi}\\
0 & 1 & 0
\end{array}\right)
\]
be the $\lambda_{1}$ perturbed matrix. Set $P_{1}=P_{2}=\mathbf{Q}_{\mathbf{1}}$
and $\pi_{0}=\pi_{\mathbf{Q}}$. The measures $\pi_{1},\pi_{2}$ are
then defined by equation (\ref{eq: Kolmogorov's consistency criteria}).
Let $m_{1}=3$ and thus $M_{1}=6.$ Set $P_{j}=\mathbf{Q}$ for $j\in\left[N_{1},M_{1}\right)=\left[3,6\right)$
and $\pi_{3},\pi_{4},\pi_{5}$ be defined by equation (\ref{eq: Kolmogorov's consistency criteria}). 

Assume that $\left\{ \lambda_{j},m_{j},n_{j},N_{j},M_{j}\right\} _{j=1}^{l-1}$
have been chosen. 

\textbf{Choice of $\lambda_{l}$:} Notice that the function 
\[
f(x):=x\frac{1+\varphi}{1+\varphi x}.
\]
 is monotone increasing and continuous in the segment $[1,\infty)$.
Therefore we can choose $\lambda_{n}$\textgreater{}1 which satisfies
the following three conditions:
\begin{enumerate}
\item Finite approximation of the Radon-Nykodym derivatives condition: 
\begin{equation}
\left(\lambda_{l}\right)^{2m_{l-1}}<e^{\frac{1}{2^{l}}}.\label{eq:condition on Lambda}
\end{equation}
This condition ensures an approximation of the derivatives by a finite
product. 
\item Lattice condition: 
\begin{equation}
\lambda_{l-1}\cdot\frac{1+\varphi}{1+\varphi\lambda_{l-1}}\in\left(\lambda_{l}\cdot\frac{1+\varphi}{1+\varphi\lambda_{l}}\right)^{\NN},\label{eq: Lattice condition on lambda}
\end{equation}
where $a^{\NN}:=\left\{ a^{n}:n\in\NN\right\} $.
\item Let 
\[
\mathbf{Q_{l}}:=\left(\begin{array}{ccc}
\frac{\varphi\lambda_{l}}{1+\varphi\lambda_{l}} & 0 & \frac{1}{1+\varphi\lambda_{l}}\\
\frac{\varphi}{1+\varphi} & 0 & \frac{1}{1+\varphi}\\
0 & 1 & 0
\end{array}\right).
\]
and $\pi_{\mathbf{Q_{l}}}$ it's unique stationary probability. Notice
that when $\lambda_{l}$ is close to $1$, then $\mathbf{Q_{l}}$
is close to $\mathbf{Q}$ in the $L_{\infty}$ sense. Therefore by
continuity of the stationary distribution we can demand that 
\begin{equation}
\|\pi_{\mathbf{Q}}-\pi_{\mathbf{Q_{l}}}\|_{\infty}<\frac{1}{2^{l}}.\label{eq: lambda is small so stationary dist. are close}
\end{equation}

\end{enumerate}
\textbf{Choice of $n_{l}$:} It follows from the Lattice condition,
equation (\ref{eq: Lattice condition on lambda}), that for each $k\leq l-1$,
\[
\left(\lambda_{k}\cdot\frac{1+\varphi}{1+\varphi\lambda_{k}}\right)\in\left(\lambda_{l}\cdot\frac{1+\varphi}{1+\varphi\lambda_{l}}\right)^{\NN}.
\]
 Choose $n_{l}$ large enough so that for every $k\leq l-1$ (notice
that the demand on $k=1$ is enough) there exists $\NN\ni p=p(k,l)\leq\frac{n_{l}}{20}$
so that 
\begin{equation}
\left(\lambda_{l}\cdot\frac{1+\varphi}{1+\varphi\lambda_{l}}\right)^{p}=\left(\lambda_{k}\cdot\frac{1+\varphi}{1+\varphi\lambda_{k}}\right).\label{eq: n_t is very large w.r.t lattic condition}
\end{equation}

Till now we have defined $\left\{ P_{j},\pi_{j}\right\} _{j=-\infty}^{M_{l-1}}$.
By the mean ergodic theorem for Markov chains \cite[Th. 4.16]{Levin Peres Wilmer Mixing Times}
and (\ref{eq: lambda is small so stationary dist. are close}), one
can demand by enlarging $n_{l}$ if necessary that in addition 
\begin{equation}
\nu_{\pi_{M_{l-1}},\mathbf{Q_{l}}}\left(x:\ \frac{1}{n_{l}}\sum_{j=1}^{n_{l}}{\bf 1}_{\left[x_{j}=1\right]}\in\left(\frac{1}{\sqrt{5}}-\frac{1}{2^{l}},\frac{1}{\sqrt{5}}+\frac{1}{2^{l}}\right)\right)>1-\frac{1}{l},\label{Ergodic theorem cond 1}
\end{equation}
and 
\begin{equation}
\nu_{\pi_{M_{l-1}},\mathbf{Q_{l}}}\left(x:\ \frac{1}{n_{l}}\sum_{j=1}^{n_{l}}{\bf 1}_{\left[x_{j}=2,x_{j+1}=3\right]}>\frac{1}{15}\right)>1-\frac{1}{l},\label{eq:Ergodic theorem cond 2}
\end{equation}
where $\nu$ is the Markov measure on $\{1,2,3\}^{\NN}$ defined by
and ${\rm \mathbf{Q_{l}}}$ and $\pi_{M_{t-1}}$. The numbers inside
the set were chosen since 
\[
\pi_{\mathbf{Q_{l}}}(1)\in\left(\frac{1}{\sqrt{5}}-\frac{1}{2^{l}},\frac{1}{\sqrt{5}}+\frac{1}{2^{l}}\right),
\]
and similarly for $l$ large enough 
\begin{eqnarray*}
\int{\bf 1}_{\left[x_{0}=2,x_{1}=3\right]}(x)d\nu_{\pi_{\mathbf{Q_{l}}},\mathbf{Q_{l}}} & = & \pi_{\mathbf{Q_{l}}}(2)\left(\mathbf{Q_{l}}\right)_{2,3}=\left(\frac{1}{\varphi\sqrt{5}}\pm\frac{1}{2^{l}}\right)\frac{1}{\varphi+1}>\frac{1}{15}.
\end{eqnarray*}
\textbf{Choice of $N_{l}$:} Let $N_{l}:=M_{l-1}+n_{l}$. Now set
for all $j\in\left[M_{l-1},N_{l}\right)$, 
\[
P_{j}=\mathbf{Q{}_{l}}
\]
 and $\left\{ \pi_{j}\right\} _{j=M_{l-1}+1}^{N_{l}}$ be defined
by equation (\ref{eq: Kolmogorov's consistency criteria}). \\
\textbf{Choice of $m_{l}$:} Let $k_{l}$ be the $\left(1\pm\left(\frac{1}{3}\right)^{3N_{l}}\right)$
mixing time of $\mathbf{Q}$. That is for every ${\bf n}>k_{l}$ ,
$A\in\F{}\left(0,l\right)$, $B\in\mathcal{F}\left(l+{\bf n},\infty\right)$
and initial distribution $\tilde{\pi}$, 
\begin{equation}
\nu_{\tilde{\pi},\mathbf{Q}}\left(A\cap B\right)=\left(1\pm3^{-3N_{l}}\right)\nu_{\tilde{\pi},\mathbf{Q}}\left(A\right)\nu_{\pi_{\mathbf{Q}},\mathbf{Q}}\left(T^{{\bf n}+l}B\right).\label{eq: Mixing time}
\end{equation}
Demand in addition that $k_{l}>N_{l}$ and 
\begin{equation}
\left\Vert \pi_{N_{l}}\mathbf{Q}^{k_{l}}-\pi_{\mathbf{Q}}\right\Vert _{\infty}<3^{-3N_{l}}.\label{eq: almost stationary dist. for Q}
\end{equation}
 To explain the last condition notice that equation (\ref{eq: Kolmogorov's consistency criteria})
together with the fact that $P_{j}$ is constant on blocks means that
\[
\pi_{N_{l}}\mathbf{Q}^{m}=\pi_{N_{l}+m}.
\]
Let $m_{l}$ be large enough so that 
\begin{equation}
\left(1-9^{-3N_{l}}\right)^{m_{l}/4k_{l}}\leq\frac{1}{l},\label{m_l large so that the sum of bad events is small}
\end{equation}
and 
\begin{equation}
\left(m_{l}-N_{l}\right)\lambda_{1}^{-2N_{l}}\geq1.\label{m_l is large enough for shift conservative}
\end{equation}

To summarize the construction. We have defined inductively sequences
$\left\{ n_{l}\right\} ,\ \left\{ N_{l}\right\} ,\ \left\{ m_{l}\right\} ,\ \left\{ M_{l}\right\} $
of integers which satisfy 
\[
M_{l}<N_{l+1}=M_{l}+n_{l}<M_{l+1}=N_{l+1}+m_{l+1}.
\]
In addition we have defined a monotone decreasing sequence $\left\{ \lambda_{l}\right\} $
which decreases to $1$ and using that sequence we defined new stochastic
matrices $\left\{ \mathbf{Q_{l}}\right\} $. Now we set 
\[
P_{j}:=\begin{cases}
\mathbf{Q}, & j\leq0\\
\mathbf{Q_{l}}, & M_{l-1}\leq j<N_{l}\\
\mathbf{Q}, & N_{l}\leq j<M_{l}
\end{cases},
\]
and $\pi_{j}=\pi_{\P}$ for $j\leq0$. The rest of the $\pi_{j}$'s
are defined by the consistency condition, equation (\ref{eq: Kolmogorov's consistency criteria}).
Finally let $\mu$ be the Markovian measure on $\left\{ 1,2,3\right\} ^{\mathbb{Z}}$
defined by $\left\{ \pi_{j},P_{j}\right\} _{j=-\infty}^{\infty}$. 

Notice that for all $j\in\NN$,  ${\rm supp}P_{j}\equiv{\rm supp}A={\rm supp}\mathbf{Q}$.

\subsubsection{Statement of the Theorem and the proof of non singularity and conservativity}
\begin{thm}
The shift $\left(\left\{ 1,2,3\right\} ^{\mathbb{Z}},\mu,T\right)$
is non singular,conservative, ergodic and of type ${\rm III}_{1}$. \end{thm}
\begin{proof}
\textbf{{[}Non Singularity and $K$ property{]}}

Since $\mu\circ T$ is the markovian measure generated by $\tilde{P}_{j}=P_{j-1}$
and $\tilde{\pi}_{j}=\pi_{j-1}$, it follows from (\ref{eq: Non singularity condition.})
and the block structure of $P_{j}$ that the shift is non singular
if and only if 
\begin{eqnarray*}
 &  & \sum_{t=1}^{\infty}\sum_{s\in S}\left\{ \left(\sqrt{P_{N_{t}}\left(x_{N_{t}},s\right)}-\sqrt{P_{N_{t}-1}\left(x_{N_{t}},s\right)}\right)^{2}\right.\\
 &  & \ \ \ \ \ \left.+\left(\sqrt{P_{M_{t}}\left(x_{M_{t}},s\right)}-\sqrt{P_{M_{t}-1}\left(x_{M_{t}},s\right)}\right)^{2}\right\} <\infty,\ \mu\circ T\ a.s.\ x.
\end{eqnarray*}
Since for all $j\in\ZZ$, $P_{j}\left(3,2\right)\equiv1,\ P_{j}\left(2,1\right)=1-P_{j}\left(2,3\right)\equiv\frac{\varphi}{1+\varphi}$,
the sum is dominated by 
\begin{eqnarray*}
\sum_{k=1}^{\infty}\sum_{s\in S}\left\{ 2\left(\sqrt{P_{N_{k}}\left(1,s\right)}-\sqrt{P_{N_{k}-1}\left(1,s\right)}\right)^{2}\right\}  & = & 2\sum_{k=1}^{\infty}\left\{ \left(\sqrt{\frac{\lambda_{k}\varphi}{1+\lambda_{k}\varphi}}-\sqrt{\frac{\varphi}{1+\varphi}}\right)^{2}\right.\\
 &  & \left.+\left(\sqrt{\frac{1}{1+\lambda_{k}\varphi}}-\sqrt{\frac{1}{1+\varphi}}\right)^{2}\right\} .
\end{eqnarray*}
This sum converges or diverges together with ${\displaystyle \sum_{k=1}^{\infty}\left|\lambda_{k}-1\right|^{2}}$.
As a consequence of condition (\ref{eq:condition on Lambda}) on $\left\{ \lambda_{j}\right\} $,
this sum is finite. Since $P_{j}\equiv\mathbf{Q}$ for all $j\leq0$,
\[
T'(x)=\frac{d\mu\circ T}{d\mu}(x)=\prod_{k=1}^{\infty}\frac{P_{k-1}\left(x_{k},x_{k+1}\right)}{P_{k}\left(x_{k},x_{k}+1\right)}.
\]
The sequence $\left\{ P_{j}\right\} _{j\in\ZZ}$ satisfies 
\[
\inf\left[\left(P_{j}P_{j+1}P_{j+2}\right)(s,t):\ s,t\in\left\{ 1,2,3\right\} ,j\in\mathbb{Z}\right]:=c>0,
\]
thus by Proposition \ref{prop:Exactness of Markov shifts} the one
sided shift $\left(\{1,2,3\}^{\NN},\mathcal{F},\mu_{+},\sigma\right)$
is an exact factor and $T'$ is $\mathcal{F}$ measurable and thus
the shift is a $K$ automorphism. Here $\mu_{+}$ denotes the measure
on the one sided shift space defined by $\left\{ \pi_{j},P_{j}\right\} _{j\geq1}$.
\end{proof}
In order to show the other properties of the Markov Shift, we will
need a more concrete expression of the Radon Nykodym derivatives.
The measure $\mu$, or more concretely it's transition matrices, differs
from the stationary $\left\{ \pi_{\mathbf{Q}},\mathbf{Q}\right\} $
measure only when one moves inside state $1$ in the segments $\left[M_{j},N_{j+1}\right)$.
Denote by 
\[
L_{j}(x):=\#\left\{ k\in\left[M_{j-1},N_{j}\right):\ x_{k}=1\right\} 
\]
and 
\[
V_{j}(x)=\#\left\{ k\in\left[M_{j-1},N_{j}\right):\ x_{k}=x_{k+1}=1\right\} .
\]

\begin{lem}
\label{lem: R_N derivatives}For every $\epsilon>0$, there exists
$t_{0}\in\NN$ s.t for every $t>t_{0}$, $N_{t}\leq n<m_{t}$ and
$x\in\{1,2,3\}^{\ZZ}$, 
\[
\left(T^{n}\right)'(x)=(1\pm\epsilon)\prod_{k=1}^{t}\left[\left(\frac{1+\varphi}{1+\varphi\lambda_{k}}\right)^{L_{k}\circ T^{n}(x)-L_{k}(x)}\cdot\lambda_{k}^{V_{k}\circ T^{n}(x)-V_{k}(x)}\right].
\]
 \end{lem}
\begin{proof}
Let $\epsilon>0$, $t\in\NN$ and $N_{t}\leq n<m_{t}$. Canceling
out all the $k's$ such that $P_{k-n}=P_{k}$ one can see that 
\begin{eqnarray*}
\left(T^{n}\right)'(x) & = & I_{t}\cdot\tilde{I_{t}}
\end{eqnarray*}
 where 
\[
I_{t}=\prod_{u=1}^{t}\left[\left(\prod_{k=M_{u-1}}^{N_{u}}\frac{P_{k-n}\left(x_{k},x_{k+1}\right)}{P_{k}\left(x_{k},x_{k+1}\right)}\right)\cdot\left(\prod_{k=M_{u-1}+n}^{N_{u}+n-1}\frac{P_{k-n}\left(x_{k},x_{k+1}\right)}{P_{k}\left(x_{k},x_{k+1}\right)}\right)\right]
\]
and 
\[
\tilde{I}_{t}=\prod_{u=t+1}^{\infty}\left[\left(\prod_{k=M_{u-1}}^{M_{u-1}+n-1}\frac{P_{k-n}\left(x_{k},x_{k+1}\right)}{P_{k}\left(x_{k},x_{k+1}\right)}\right)\cdot\left(\prod_{k=N_{u}}^{N_{u}+n-1}\frac{P_{k-n}\left(x_{k},x_{k+1}\right)}{P_{k}\left(x_{k},x_{k+1}\right)}\right)\right].
\]
We will analyze the two terms separately. Since for every $M_{u-1}\leq k<M_{u-1}+n$,
$P_{k}=\mathbf{Q_{u}}$ and $P_{k-n}=\mathbf{Q}$, 
\[
\frac{P_{k-n}\left(x_{k},x_{k+1}\right)}{P_{k}\left(x_{k},x_{k+1}\right)}\leq\frac{{\rm \mathbf{Q}_{1,3}}}{{\rm \left(\mathbf{Q_{u}}\right)}_{1,3}}=\frac{1+\varphi\lambda_{u}}{1+\varphi}\leq\lambda_{u}.
\]
Similarly for $N_{u}\leq k<N_{u}+n$, $P^{(k)}=\mathbf{Q}$ and $P_{k-n}={\rm \mathbf{Q_{u}}}$
. Therefore 
\[
\frac{P_{k-n}\left(x_{k},x_{k+1}\right)}{P_{k}\left(x_{k},x_{k+1}\right)}\leq\frac{{\rm \left(\mathbf{Q_{u}}\right)_{1,1}}}{\mathbf{Q}_{1,1}}\leq\lambda_{u}.
\]
and 
\[
\lambda_{u}^{-2n}\leq\left(\prod_{k=M_{u-1}}^{M_{u-1}+n-1}\frac{P_{k-n}\left(x_{k},x_{k+1}\right)}{P_{k}\left(x_{k},x_{k+1}\right)}\right)\cdot\left(\prod_{k=N_{u}}^{N_{u}+n-1}\frac{P_{k-n}\left(x_{k},x_{k+1}\right)}{P_{k}\left(x_{k},x_{k+1}\right)}\right)\leq\lambda_{u}^{2n},
\]
here the lower bound is achieved by a similar analysis. This gives
\begin{eqnarray*}
\tilde{I}_{t} & = & \prod_{u=t+1}^{\infty}\left[\lambda_{u}^{\pm2n}\right]\\
 & = & \prod_{u=t+1}^{\infty}\left[\lambda_{u}^{\pm2m_{u-1}}\right]\quad\text{(since\ \ensuremath{\forall u>t},\ \ensuremath{n<m_{t}<m_{u})}}\\
 & \overset{\eqref{eq:condition on Lambda}}{=} & e^{\pm\sum_{n=t+1}^{\infty}\frac{1}{2^{n}}}\xrightarrow[t\to\infty]{}1.
\end{eqnarray*}
Consequently there exists $t_{0}\in\NN$ so that for all $x\in\Sigma_{{\bf A}}$,
for all $t>t_{0}$ and $N_{t}\leq n\leq m_{t}$,
\[
\left(T^{n}\right)'(x)=(1\pm\epsilon)I_{t}.
\]
By noticing that for $k\in\bigcup_{j=1}^{t}\left(\left[M_{j-1},N_{j}\right)\cup\left[M_{j-1}+n,N_{j}+n\right)\right),$
\[
P_{k-n}\left(x_{k},x_{k+1}\right)\neq P_{k}\left(x_{k},x_{k+1}\right)
\]
if and only if $x_{k}=1$ one can check that 
\[
I_{t}=\prod_{k=1}^{t}\left[\left(\frac{1+\varphi}{1+\varphi\lambda_{k}}\right)^{L_{k}\circ T^{n}(x)-L_{k}(x)}\lambda_{k}^{V_{k}\circ T^{n}(x)-V_{k}(x)}\right].
\]
\end{proof}
\begin{cor}
The shift $\left(\{1,2,3\}^{\ZZ},\mu,T\right)$ is conservative and
ergodic.\end{cor}
\begin{proof}
Since the shift is a $K$-automorphism it is enough to prove conservativity. 

For every $j\in\NN$, $0\leq L_{k}(x),D_{k}(x)\leq n_{k}$. Whence
\begin{eqnarray*}
\left(\frac{1+\varphi}{1+\varphi\lambda_{k}}\right)^{L_{k}\circ T^{n}(x)-L_{k}(x)}\lambda_{k}^{V_{k}\circ T^{n}(x)-V_{k}(x)} & \geq & \left(\frac{1+\varphi}{1+\varphi\lambda_{k}}\right)^{L_{k}\circ T^{n}(x)}\lambda_{k}^{-V_{k}(x)}\\
 & \geq & \lambda_{k}^{-2n_{k}}\geq\lambda_{1}^{-2n_{k}},
\end{eqnarray*}
and for every $t\in\NN$, 
\[
\prod_{k=1}^{t}\left[\left(\frac{1+\varphi}{1+\varphi\lambda_{k}}\right)^{L_{k}\circ T^{n}(x)-L_{k}(x)}\lambda_{k}^{V_{k}\circ T^{n}(x)-V_{k}(x)}\right]\geq\lambda_{1}^{-2\sum_{k=1}^{t}n_{k}}\geq\lambda_{1}^{-2N_{t}}.
\]
By Lemma \ref{lem: R_N derivatives} there exists $t_{0}\in\NN$ such
that for all $t\geq t_{0}$, $N_{t}\leq n\leq m_{t}$ and $x\in\Sigma_{{\bf A}}$,
$\RN n(x)\geq\frac{\lambda_{1}^{-2N_{t}}}{2}$. Therefore for all
$x\in\Sigma_{{\bf A}}$,
\[
\sum_{n=1}^{\infty}\RN n(x)\geq\sum_{t=1}^{\infty}\sum_{n=N_{t}}^{m_{t}}\RN n(x)\geq\sum_{t=t_{0}}^{\infty}\frac{1}{2}\left(m_{t}-N_{t}\right)\lambda_{1}^{-2N_{t}}\overset{\eqref{m_l is large enough for shift conservative}}{=}\infty.
\]
By Hopf's criteria the shift is conservative. 
\end{proof}

\subsubsection{Proof of the type ${\rm III}_{1}$ property}

In order to prove that the ratio set is $[0,\infty)$ we are going
to use the following principle: since $R(T)$ is a multiplicative
subset it is enough to show that there exists $y_{n}\in R(T)\backslash\{1\}$
with $y_{n}\to1$ as $n\to\infty$. 
\begin{thm}
\label{thm: lambdas are in R(T)}Let $\mu$ be the Markovian measure
constructed in Subsection \ref{sub:The-construction.}. For every
$n\in\NN$, $\lambda_{n}\cdot\frac{1+\varphi}{1+\varphi\lambda_{n}}\in R(T)$
and therefore the shift is type ${\rm III}_{1}$.
\end{thm}
Fix $n\in\NN$. The first stage in proving that $\lambda_{n}\cdot\frac{1+\varphi}{1+\varphi\lambda_{n}}\in R(T)$
is to show that the ratio set condition is satisfied for all cylinders
with a positive proportion of the measure of the cylinder set. Then
for a general $A\in\BB_{+}$, we use the density of cylinder sets
in $\mathcal{B}$. 

Given $t\in\NN$, denote by $\mathcal{C}(t)$ the collection of all
$[c]_{0}^{N_{t}}$ cylinder sets such that
\begin{equation}
L_{t}(c)=\sum_{k=M_{t-1}}^{N_{t}-1}{\bf 1}_{\left[c_{k}=1\right]}\in\left(\frac{n_{t}}{4},\frac{n_{t}}{2}\right)\text{ and}\ \sum_{k=M_{t-1}}^{N_{t}-1}{\bf 1}_{\left[c_{k}=2,c_{k+1}=3\right]}\geq\frac{n_{t}}{15}.\label{eq: Free bits in C}
\end{equation}
Since 
\[
\mu\left([c]_{M_{t-1}}^{N_{t}}\right)=\nu_{\pi_{M_{t-1}},\mathbf{Q_{t}}}\left([c]_{0}^{n_{t}}\right),
\]
it follows from (\ref{Ergodic theorem cond 1}) and (\ref{eq:Ergodic theorem cond 2})
that for all $t$ large enough, 
\[
\mu\left(\bigcup_{C\in\mathcal{C}(t)}C\right)\geq1-\frac{1}{2t}.
\]
 In order to shorten the notation, given $M,j\in\NN$, $B\in\BB$
and $\epsilon>0$, let 
\[
\mathfrak{RSC}\left(M,B,j,\epsilon\right):=B\cap T^{-M}B\cap\left[\RN M=\lambda_{j}\cdot\frac{1+\varphi}{1+\varphi\lambda_{j}}\cdot\left(1\pm\epsilon\right)\right],
\]
and for $M\in\NN$, 
\[
\Sigma_{{\bf A}}\left(M\right):=\{1,2,3\}^{M}\cap\Sigma_{{\bf A}}.
\]
 
\begin{lem}
\label{Main step for EVC. }For every $[b]_{-n}^{n}$ cylinder set,
$\epsilon>0$ and $j\in\NN$, there exists a $t_{0}\in\NN$ so that
for all $t>t_{0}$ the following holds: 

For every $C=[c]_{0}^{N_{t}-1}\in\mathcal{C}(t)$ there exists $d=d(b,C)\in\Sigma_{{\bf A}}\left(N_{t}+n\right)$
such that for every $\NN\ni l\leq m_{t}/k_{t}$, 
\begin{equation}
C\cap[d]_{lk_{t}-n}^{lk_{t}+N_{t}-1}\subset T^{-lk_{t}}[b]_{-n}^{n}\cap\left[\RN{lk_{t}}=\lambda_{j}\cdot\frac{1+\varphi}{1+\varphi\lambda_{j}}\cdot\left(1\pm\epsilon\right)\right].\label{eq: cylinder for EVC}
\end{equation}
 
\end{lem}
Recall that $k_{t}>N_{t}$ is defined as a $\left(1\pm3^{-3N_{t}}\right)$
mixing time for ${\bf Q}$. 
\begin{proof}
Let $[b]_{-n}^{n},\epsilon>0$ and $j\in\NN$ be given. By Lemma \ref{lem: R_N derivatives}
there exists $\tau$ such that for every $t\geq\tau$ and $1\leq l\leq m_{t}/k_{t}$
(here $lk_{t}\in\left[N_{t},m_{t}\right)$),
\[
\RN{lk_{t}}(x)=(1\pm\epsilon)\prod_{k=1}^{t}\left[\left(\frac{1+\varphi}{1+\varphi\lambda_{k}}\right)^{L_{k}\circ T^{lk_{t}}(x)-L_{k}(x)}\cdot\lambda_{k}^{V_{k}\circ T^{lk_{t}}(x)-V_{k}(x)}\right].
\]
Choose $t_{0}$ to be any integer which satisfies $t_{0}>\max\left(\tau,j\right)$
and $M_{t_{0}}>n$. 

Let $t>t_{0}$ and choose a cylinder set $[c]_{0}^{N_{t}}\in\mathcal{C}(t)$
which intersects $[b]_{-n}^{n}.$ That is $c_{i}=b_{i}$ for $i\in[0,n]$.
We need now to choose $d\in\Sigma_{{\bf A}}\left(N_{t}+n\right)$
which satisfies (\ref{eq: cylinder for EVC}). Notice that for $x\in[d]_{lk_{t}-n}^{lk_{t}+N_{t}}\cap[c]_{0}^{N_{t}},$
\begin{eqnarray*}
 &  & \prod_{k=1}^{t}\left[\left(\frac{1+\varphi}{1+\varphi\lambda_{k}}\right)^{L_{k}\circ T^{lk_{t}}(x)-L_{k}(x)}\cdot\lambda_{k}^{V_{k}\circ T^{lk_{t}}(x)-V_{k}(x)}\right]\\
 & = & \prod_{k=1}^{t}\left[\left(\frac{1+\varphi}{1+\varphi\lambda_{k}}\right)^{L_{k}(d)-L_{k}(c)}\cdot\lambda_{k}^{V_{k}(d)-V_{k}(c)}\right],
\end{eqnarray*}
in this representation we look at $[d]_{-n}^{N_{t}}$. For all $k\in[0,M_{t-1}]$,
let 
\[
d_{k}=c_{k}
\]
and for all $k\in[-n,0)$, 
\[
d_{k}=b_{k}.
\]
Notice that this means that for $k\in[-n,n]$, $d_{k}=b_{k}$ and
thus 
\[
[d]_{lk_{t}-n}^{lk_{t}+N_{t}}\subset T^{-lk_{t}}[b]_{-n}^{n}.
\]
 Let $p(j,t)\leq\frac{n_{t}}{20}$ be the integer (condition (\ref{eq: n_t is very large w.r.t lattic condition}))
such that 
\[
\left(\lambda_{t}\cdot\frac{1+\varphi}{1+\varphi\lambda_{t}}\right)^{p(j,t)}=\lambda_{j}\cdot\frac{1+\varphi}{1+\varphi\lambda_{j}}.
\]
Set $d_{k}=1$ for all $k\in\left[M_{t-1},M_{t-1}+V_{t}(c)+p(j,t)\right]$
and then continue repeatedly with the sequence $"321"$ $L_{t}(c)-V_{t}(c)$
times. Since $c$ satisfies (\ref{eq: Free bits in C}), this construction
is well defined (e.g. we have not reached yet $k=N_{t}-1$). Continue
with sequences of $32$ till $k=N_{t}-1$. 

Thus we have defined $d$ in such a way so that 
\[
L_{t}(d)-L_{t}(c)=p(j,t)
\]
and 
\[
V_{t}(d)-V_{t}(c)=p(j,t).
\]
In addition for all $0\leq k<t$,
\[
L_{k}(d)=L_{k}(c)\ {\rm and}\ V_{k}(c)=V_{k}(d).
\]
 Thus for all $x\in[d]_{lk_{t}-n}^{lk_{t}+N_{t}}\cap[c]_{0}^{N_{t}},$
\begin{eqnarray*}
\RN{lk_{t}}(x) & = & \left(1\pm\epsilon\right)\prod_{k=1}^{t}\left[\left(\frac{1+\varphi}{1+\varphi\lambda_{k}}\right)^{L_{k}(d)-L_{k}(c)}\cdot\lambda_{k}^{V_{k}(d)-V_{k}(c)}\right]\\
 & = & \left(1\pm\epsilon\right)\left(\lambda_{t}\cdot\frac{1+\varphi}{1+\varphi\lambda_{t}}\right)^{p(j,t)}\\
 & = & \left(1\pm\epsilon\right)\left(\lambda_{j}\frac{1+\varphi}{1+\varphi\lambda_{j}}\right).
\end{eqnarray*}
This proves the lemma. 
\end{proof}
In the course of the proof one sees that the event 
\[
\left([b]_{-n}^{n}\cap[c]_{0}^{N_{t}}\right)\cap\left(T^{-lk_{t}}[b]_{-n}^{n}\cap\left\{ \RN{lk_{t}}=\lambda_{j}\cdot\frac{1+\varphi}{1+\varphi\lambda_{j}}\cdot\left(1\pm\epsilon\right)\right\} \right)^{c}
\]
is $\F{}\left(lk_{t}-n,lk_{t}+N_{t}\right)$ measurable and does not
depend on $\cF\left(lk_{t}+N_{t},lk_{t}+2N_{t}\right)$. 
\begin{rem}
Given $[c]_{0}^{N_{t}}\in\mathcal{C}_{t}$ we have defined $d=d(c)\in\Sigma_{{\bf A}}\left(N_{t}+n\right)$.
The definition of $d$ is not necessarily one to one. This is because
if $\left[\tilde{c}\right]_{0}^{M_{t-1}}=[c]_{0}^{M_{t-1}}$, $V_{t}(c)=V_{t}\left(\tilde{c}\right)$
and $L_{t}(c)=L_{t}\left(\tilde{c}\right)$ then $d(c)=d\left(\tilde{c}\right)$.
In order to make it one to one we will use 
\[
[d(c),c]_{lk_{t}-n}^{lk_{t}+2N_{t}}
\]
instead of $[d(c)]_{lk_{t}}^{lk_{t}+N_{t}}$where by $\left[a,b\right]_{l}^{l+{\rm length}(a)+{\rm length}(b)}$
we mean the concatenation of $a$ and $b$. This can be thought of
as putting a Marker on $d(c)$. In order that the concatenation will
be in $\Sigma_{A}$ we need that 
\[
{\bf Q}\left(d(c)_{N_{t}-1},c_{0}\right)>0.
\]
This can be done by possibly changing the last two coordinates of
$d(c)$. This will change the value of $\RN{lk_{t}}$ by at most a
factor of $\lambda_{t}^{\pm4}$ , which is close enough to one. We
will denote by ${\bf d}(c):=(\tilde{d}(c),c)$. We still have 
\[
[\mathbf{d}(c)]_{lk_{t}-n}^{lk_{t}+2N_{t}}\subset T^{-lk_{t}}[b]_{-n}^{n}\cap\left[\RN{lk_{t}}=\lambda_{j}\cdot\frac{1+\varphi}{1+\varphi\lambda_{j}}\cdot\left(1\pm\epsilon\right)\right],
\]
but now the map $c\mapsto\mathbf{d}(c)$ is one to one. 
\end{rem}
In the proof of the next lemma we will make use of the fact that for
every cylinder set $\left([a]_{m}^{l}\right)^{c}$ is $\cF(m,l)$
measurable. 
\begin{lem}
\label{EVC condition}For every $[b]_{-n}^{n}$ cylinder set, $\epsilon>0$
and $j\in\NN$ there exists $t_{0}\in\NN$ such that for all $t>t_{0}$,
\[
\mu\left(\bigcup_{l=1}^{m_{t}/4k_{t}}\mathfrak{RSC}\left(4lk_{t},[b]_{-n}^{n},j,\epsilon\right)\right)\geq0.8\mu\left([b]_{-n}^{n}\right).
\]
\end{lem}
\begin{proof}
Let $[b]_{-n}^{n}$ be a cylinder set and $t_{0}$ be as in Lemma
\ref{Main step for EVC. }. For all $t\geq t_{0}$, $[c]_{0}^{N_{t}}\in\mathcal{C}(t)$
which intersects $[b]_{-n}^{n}$ and $1\leq l\leq m_{t}/4k_{t}$,
\[
\left([c]_{0}^{N_{t}}\cap[b]_{-n}^{n}\right)\cap\left(\mathfrak{RSC}\left(4lk_{t},[b]_{-n}^{n},j,\epsilon\right)\right)^{c}\subset[c]_{0}^{N_{T}}\cap[b]_{-n}^{n}\cap\left([\mathbf{d}(c)]_{4lk_{t}-n}^{4lk_{t}+N_{t}}\right)^{c}
\]
The fact that $P_{j}\equiv\mathbf{Q}$ for $j\in\left[N_{t},M_{t}\right)$
implies that 
\begin{eqnarray*}
\mu\left([d(c)]_{4lk_{t}-n}^{4lk_{t}+2N_{t}}\right) & = & \nu_{\pi_{4lk_{t}-n},\mathbf{Q}}\left([d(c)]_{0}^{4lk_{t}+2N_{t}+n}\right)\\
 & \geq & \pi_{4lk_{t}-n}\left(d_{-n}\right)\mathbf{Q}_{1,3}^{2N_{t}+n-1}\\
 & \geq & \left(\pi_{\mathbf{Q}}\left(d_{-n}\right)-\frac{1}{3^{N_{t}}}\right)\mathbf{Q}_{1,3}^{2N_{t}+n-1}\ \ (\text{by\ \eqref{eq: almost stationary dist. for Q})}\\
 & \gtrsim & \frac{1}{3^{3N_{t}}}.
\end{eqnarray*}
Therefore, one has by many application of (\ref{eq: Mixing time})
(mixing time condition), 
\begin{eqnarray*}
 &  & \mu\left(\left([b]_{-n}^{n}\cap[c]_{0}^{N_{t}}\right)\cap\left(\bigcup_{l=1}^{m_{t}/4k_{t}}\mathfrak{RSC}\left(lk_{t},[b]_{-n}^{n},j,\epsilon\right)\right)^{c}\right)\\
 & \leq & \mu\left(\left([b]_{-n}^{n}\cap[c]_{0}^{N_{t}}\right)\cap\left\{ \bigcap_{l=1}^{m_{t}/4k_{t}}\left([d_{c}]_{4lk_{t}}^{4lk_{t}+N_{t}}\right)^{c}\right\} \right)\\
 & \leq & \mu\left(\left([b]_{-n}^{n}\cap[c]_{0}^{N_{t}}\right)\right)\prod_{1=1}^{m_{t}/4k_{t}}\left[\left(1+3^{-3N_{t}}\right)\left(1-\nu_{\pi_{{\bf Q}},{\bf Q}}\left([d_{c}]_{4lk_{t}}^{4lk_{t}+N_{t}}\right)\right)\right]\\
 & \leq & \mu\left(\left([b]_{-n}^{n}\cap[c]_{0}^{N_{t}}\right)\right)\left[\left(1+3^{-3N_{t}}\right)\left(1-3^{-3N_{t}}\right)\right]^{m_{t}/4k_{t}}\\
 & \overset{\eqref{m_l large so that the sum of bad events is small}}{\leq} & \frac{1}{t}\mu\left(\left([b]_{-n}^{n}\cap[c]_{0}^{N_{t}}\right)\right).
\end{eqnarray*}
Notice that we used the fact that 
\[
\left(4\left(l+1\right)k_{t}-n\right)-\left(4lk_{t}+2N_{t}\right)>(4l+3)k_{t}-\left(4l+2\right)k_{t}=k_{t}.
\]
If $t$ is large enough then 
\[
\mu\left(\Sigma_{{\bf A}}\backslash\bigcup_{C\in\mathcal{C}(t)}C\right)<0.1\mu\left([b]_{-n}^{n}\right),
\]
and for all $[c]_{0}^{N_{t}}=C\in\mathcal{C}(t)$, 
\begin{eqnarray*}
\mu\left([b]_{-n}^{n}\cap[c]_{0}^{N_{t}}\cap\left(\bigcup_{l=1}^{m_{t}/4k_{t}}\mathfrak{RSC}\left(lk_{t},[b]_{-n}^{n},j,\epsilon\right)\right)\right) & > & \left(1-\frac{1}{t}\right)\mu\left([b]_{-n}^{n}\cap[c]_{0}^{N_{t}}\right)\\
 & \geq & 0.9\mu\left([b]_{-n}^{n}\cap[c]_{0}^{N_{t}}\right).
\end{eqnarray*}
The Lemma follows from 
\begin{eqnarray*}
 &  & \mu\left(\bigcup_{l=1}^{m_{t}/4k_{t}}\mathfrak{RSC}\left(lk_{t},[b]_{-n}^{n},j,\epsilon\right)\right)\\
 & \geq & \mu\left(\uplus_{[c]_{0}^{N_{t}}\in\mathcal{C}(t)}[b]_{-n}^{n}\cap[c]_{0}^{N_{t}}\cap\bigcup_{l=1}^{m_{t}/4k_{t}}\mathfrak{RSC}\left(lk_{t},[b]_{-n}^{n},j,\epsilon\right)\right)\\
 & \geq & 0.9\sum_{[c]_{0}^{N_{T}}\in\mathcal{C}(t)}\mu\left([b]_{-n}^{n}\cap[c]_{0}^{N_{t}}\right)\\
 & \geq & 0.8\mu\left(\left[b\right]_{-n}^{n}\right)
\end{eqnarray*}

\end{proof}
\begin{proof}[Proof of Theorem \ref{thm: lambdas are in R(T)}] 

This is a standard approximation technique. Let $j\in\NN$, $A\in\BB$,
$\mu(A)>0$ and $\epsilon>0$. Since the ratio set condition on the
derivative is monotone with respect to $\epsilon$ and

\[
1<\lambda_{j}\cdot\frac{1+\varphi}{1+\varphi\lambda_{j}}<2,
\]
 we can assume that 
\begin{equation}
1\leq\lambda_{j}\cdot\frac{1+\varphi}{1+\varphi\lambda_{j}}(1\pm\epsilon)\leq2.\label{eq:epsilon is small enough}
\end{equation}
 Since $\cF(-n,n)\uparrow\BB$ as $n\to\infty$, there exists a cylinder
set $\mathfrak{b}=[b]_{-n}^{n}$ such that 
\[
\mu\left(A\cap\mathfrak{b}\right)>0.99\mu\left(\mathfrak{b}\right).
\]
 By Lemma \ref{EVC condition} there exists $t\in\NN$ for which 
\[
\mu\left(\mathfrak{b\cap}\left\{ \bigcup_{l=1}^{m_{t}/4k_{t}}T^{-4lk_{t}}\mathfrak{b}\cap\left[\RN{4lk{}_{t}}=\lambda_{j}\cdot\frac{1+\varphi}{1+\varphi\lambda_{j}}\cdot\left(1\pm\epsilon\right)\right]\right\} \right)>0.8\mu(\mathfrak{b}).
\]
Denote by 
\[
B=\mathfrak{b\cap}\left\{ \bigcup_{l=1}^{m_{t}/4k_{t}}T^{-4lk_{t}}\mathfrak{b}\cap\left[\RN{4lk{}_{t}}=\lambda_{j}\cdot\frac{1+\varphi}{1+\varphi\lambda_{j}}\cdot\left(1\pm\epsilon\right)\right]\right\} .
\]
We can assume that for $x\in B$, there exists $C(x)=[c]_{0}^{N_{t}}\in\mathcal{C}_{t}$
so that $x\in C(x)$. Then by the proof of Lemma \ref{Main step for EVC. }
there exists $d(C(x))\in\Sigma_{{\bf A}}\left(2N_{t}+n\right)$ such
that if $x\in[d(C(x))]_{4lk_{t}-n}^{4lk_{t}+2N_{t}}$, then 
\begin{equation}
\RN{4lk_{t}}(x)=\lambda_{j}\cdot\frac{1+\varphi}{1+\varphi\lambda_{j}}\cdot\left(1\pm\epsilon\right)\ \text{and}\ x\in T^{-4lk_{t}}\mathfrak{b}.\label{eq: RN on D(C(x))}
\end{equation}
 Define $\phi:B\to\NN$ 
\[
\phi(x):=\inf\left\{ l\leq m_{t}/4k_{t}:\ [x]_{4lk_{t}-n}^{4lk_{t}+2N_{t}}=[{\bf d}(C(x))]_{4lk_{t}-n}^{4lk_{t}+2N_{t}}\right\} 
\]
and $S=T^{\phi}:B\to S(B)\subset\mathfrak{b}$. We claim that $S$
is one to one. Indeed, since the map $[c]_{0}^{N_{t}}\mapsto{\bf d}(c)$
is one to one, for every $x,y\in B$ such that $C(x)\neq C(y)$,
\[
[Sy]_{-n}^{2N_{t}}=[{\bf d}(C(y))]_{-n}^{2N_{t}}\neq[{\bf d}(C(x))]_{-n}^{2N_{t}}=[Sx]_{-n}^{2N_{t}},
\]
consequently $Sx\neq Sy$. In addition, by the definition of $\phi$,
if $x\neq y$ and $C(x)=C(y)$ then $Sx\neq Sy$. 

It follows from (\ref{eq: RN on D(C(x))}) and (\ref{eq:epsilon is small enough}),
that for all $x\in B$, 
\[
S'(x):=\frac{d\mu\circ S}{d\mu}(x)=\lambda_{j}\cdot\frac{1+\varphi}{1+\varphi\lambda_{j}}\cdot\left(1\pm\epsilon\right)\in\left[1,2\right].
\]
Therefore $\frac{d\mu\circ S^{-1}}{d\mu}(y)\geq\frac{1}{2}$ for all
$y\in S(B)$. A calculation shows that 
\begin{eqnarray*}
\mu(S(B)\cap A) & > & \mu(S(B))-\mu(\mathfrak{b}\backslash A)\\
 & > & \mu(B)-\mu(\mathfrak{b}\backslash A)\\
 & = & 0.79\mu\left(\mathfrak{b}\right),
\end{eqnarray*}
and 
\[
\mu\left(S^{-1}\left(S(B)\cap A\right)\right)>\frac{\mu\left(S(B)\cap A\right)}{2}>0.39\mu(\mathfrak{b}).
\]
So 
\begin{eqnarray*}
 &  & \sum_{l=1}^{m_{t}/4lk_{t}}\mu\left(A\cap\left\{ T^{-4lk_{t}}A\cap\left[\RN{4lk{}_{t}}=\lambda_{j}\cdot\frac{1+\varphi}{1+\varphi\lambda_{j}}\cdot\left(1\pm\epsilon\right)\right]\right\} \cap\left[\phi=4lk_{t}\right]\right)\\
 & \geq & \mu\left(\left(B\cap A\right)\cap S^{-1}\left(S(B)\cap A\right)\right)\ \ \ \ \ \left\{ \text{Notice that }B,\ S(B)\subset\mathfrak{b}\right\} \\
 & \geq & \mu\left(B\cap A\right)-\mu\left(\mathfrak{b}\backslash S^{-1}\left(S(B)\cap A\right)\right)\\
 & \geq & 0.18\mu\left(\mathfrak{b}\right),
\end{eqnarray*}
and thus there exists $l\in\NN$ such that 
\[
\mu\left(A\cap T^{-4lk_{t}}A\cap\left[\RN{4lk_{t}}=\lambda_{j}\cdot\frac{1+\varphi}{1+\varphi\lambda_{j}}\cdot\left(1\pm\epsilon\right)\right]\right)>0.
\]
This proves the Theorem. 

\end{proof}

\section{Concluding Remarks}

One feature of this construction is that if $f(x,y)=(x+y,x)$ , $\Phi:\Sigma_{A}\to\mathbb{T}^{2}$
is the semi conjugacy map constructed from the Markov partition $\left\{ R_{1},R_{2},R_{3}\right\} $
and $\nu=\mu\circ\Phi^{-1}$, then:
\begin{itemize}
\item $\Phi:\ \left(\Sigma_{A},\mu,T\right)\to\left(\mathbb{T}^{2},\nu,f\right)$
is one to one, hence a measure theoretic isomorphism.
\end{itemize}
For every $n_{k}<n_{k-1}<\cdots<n_{1}<0$ and $i_{1},...,i_{k}\in\{1,2,3\}$,
\[
\nu\left(\bigcap_{j=1}^{k}f^{-n_{j}}R_{i_{k}}\right)=\text{Leb}\left(\bigcap_{j=1}^{k}f^{-n_{j}}R_{i_{k}}\right).
\]

\begin{rem}
Given a mixing TMS $\Sigma_{A}$, one can construct a type ${\rm III}_{1}$
Markov shift supported on $\Sigma_{A}$ as follows. Define ${\rm Q}$
to be the matrix with 
\[
Q_{i,j}=\begin{cases}
1/\sum_{l=1}^{n}A_{i,l}, & A_{i,j}=1\\
0 & A_{i,j}=0
\end{cases}.
\]
If $A$ has a row $i\in\{1,..,|S|\}$ with at least two $1$'s, one
can proceed as in our example to define ${\rm Q}_{l}$ to be the matrix
${\rm Q}$ perturbed in the $i$-th row between two non zero coordinates.
The rest of the proof remains the same. 
\end{rem}
We end this section with the following open question: Given a mixing
TMS $\Sigma\subset F^{\ZZ}$ with $F$ finite. Does there exist a
Markov measure $\nu=\M{\pi_{k},R_{k}}$ so that:
\begin{itemize}
\item $\nu\circ T\sim\nu$. 
\item $\nu$ is fully supported on $\Sigma$ and the shift is ergodic with
respect to $\nu$. 
\item The non singular Markov shift $\left(\Sigma,\BB,\nu,T\right)$ is
of type ${\rm II}_{\infty}$ (preserves an a.c.i.m but no a.c.i.p.)
or ${\rm III_{\lambda}},\ 0\leq\lambda<1$. 
\item Preferably $\nu$ is half stationary meaning that there exists an
irreducible and aperiodic stochastic Matrix $R$ so that for all $k<0$,
$R_{k}=R$ and $\pi_{k}=\pi_{R}$. 
\end{itemize}
If in addition $\Sigma$ is a $TMS$ arising from a hyperbolic Toral
automorphism and $R$ is the stochastic matrix representing the Lebesgue
measure then a positive answer to this question may give new examples
of Anosov diffeomorphisms by the methods of \cite{Kos1}. In the case
where $\Sigma=\{0,1\}^{\ZZ}$ and $\nu$ is a half stationary product
measure it was shown in \cite{Kos2} that the shift is either of type
${\rm III}_{1}$, dissipative or equivalent to a classical Bernoulli
shift ($\nu$ is equivalent to a product measure with i.i.d entries).


\begin{thebibliography}{Aar1}
\bibitem[Aar1]{Aar}J. Aaronson, An introduction to infinite ergodic
theory, Amer. Math. Soc., Providence, R.I., 1997. 

\bibitem[ALV]{ALV}J. Aaronson, M. Lema\'{n}czyk, D.Volný. A cut salad
of cocycles. Fund. Math. 157 (1998), no. 2-3, 99\textendash 119.

\bibitem[Adl]{Adler} R.L. Adler. Symbolic dynamics and Markov partitions.
Bull. Amer. Math. Soc. (N.S.) 35 (1998), no. 1, 1\textendash 56. 

\bibitem[AW]{Adler Weiss} R.L. Adler, B. Weiss. Similarity of automorphisms
of the torus. Memoirs of the American Mathematical Society, No. 98
American Mathematical Society, Providence, R.I. 1970 ii+43 pp.

\bibitem[B]{Bowen Rufus} R. Bowen. Equilibrium states and the ergodic
theory of Anosov diffeomorphisms. Second revised edition. With a preface
by David Ruelle. Edited by Jean-René Chazottes. Lecture Notes in Mathematics,
470. Springer-Verlag, Berlin, 2008.


\bibitem[Ha]{Halmos} P. Halmos. Invariant measures. Ann. of Math.
(2) 48, (1947). 735\textendash 754. 

\bibitem[Ham]{Ham}T. Hamachi, On a Bernoulli shift with non-identical
factor measures, Ergod. Th. \& Dynam. Sys. 1 (1981), pp 273\textendash 283,
MR 83i:28025, Zbl 597.28022. 

\bibitem[Kos]{Kos} Z.Kosloff. On a type III1 Bernoulli shift. Ergodic
Theory Dynam. Systems 31 (2011), no. 6, 1727\textendash 1743.

\bibitem[Kos1]{Kos1} Z. Kosloff. Conservative Anosov diffeomorphisms
of the two torus without an absolutely continuous invariant measure.
Preprint.

\bibitem[Kos2]{Kos2} Z. Kosloff. On the K property for Maharam extensions
of Bernoulli shifts and a question of Krengel. Israel J. Math. 199
(2014), no. 1, 485\textendash 506.

\bibitem[Kre]{Krengel} U. Krengel, Transformations without finite
invariant measure have finite strong generators. 1970 Contributions
to Ergodic Theory and Probability (Proc. Conf., Ohio State Univ.,
Columbus, Ohio, 1970) pp. 133\textendash 157 Springer.

\bibitem[Kri]{Kri} W. Krieger, On non-singular transformations of
a measure space. I, II. Z. Wahrscheinlichkeitstheorie und Verw. Gebiete
11 (1969), 83-97

\bibitem[LM]{Le Page Mandrelkar } R. LePage, V. Mandrekar. On likelihood
ratios of measures given by Markov chains. Proc. Amer. Math. Soc.
52 (1975), 377\textendash 380.

\bibitem[LPW]{Levin Peres Wilmer Mixing Times} D.A. Levin, Y. Peres,
E.L. Wilmer. Markov chains and mixing times. With a chapter by James
G. Propp and David B. Wilson. American Mathematical Society, Providence,
RI, 2009.

\bibitem[Si]{Si} J.G. Sinai, Markov partitions and $U$-diffeomorphisms.
(Russian) Funkcional. Anal. i Priložen 2 1968 no. 1, 64\textendash 89.

\bibitem[Shi]{Shi} Shiryayev, A. N. Probability, second edition.
Translated from the Russian by R. P. Boas. Graduate Texts in Mathematics,
95. Springer-Verlag, New York, 1984.

\bibitem[ST]{S-T} C. E. Silva and P. Thieullen, A skew product entropy
for nonsingular transformations, J. Lon. Math. Soc. (2) 52 (1995),
497\textendash 516. 

\bibitem[Th]{Thaler} M. Thaler. Transformations on {[}0,1{]} with
infinite invariant measures. Israel J. Math. 46 (1983), no. 1-2, 67\textendash 96.\end{thebibliography}
\end{document}